\documentclass[a4paper,11pt]{article}

\pdfoutput=1

\usepackage{amsmath}
\usepackage{amsthm}
\usepackage{amssymb}
\usepackage[a4paper]{geometry}
\usepackage{graphicx}
\usepackage{microtype}
\usepackage{enumitem}
\usepackage{color}
\usepackage{relsize}

\newcommand{\K}[1]{\mathbb{#1}}
\newcommand{\txt}[1]{\text{\normalfont{#1}}}

\newtheorem{thm}{Theorem}
\newtheorem*{thm*}{Theorem}

\newtheorem{remark}{Remark}
\newtheorem{propo}{Proposition}
\newtheorem{lemma}{Lemma}



\title{On fully real eigenconfigurations of tensors}
\author{Khazhgali Kozhasov}
\date{}

\begin{document}
\maketitle

\begin{changemargin}{1.5cm}{1.5cm}
{\bf\noindent Abstract.}
We construct generic real symmetric tensors with only real eigenvectors or, equivalently, real homogeneous polynomials with the maximum possible finite number of critical points on the sphere.
\end{changemargin}

\section*{Introduction}

In $2005$ Lim \cite{Lim} and Qi \cite{Qi} independently initiated the spectral theory of high order tensors. They introduced several generalizations of the classical concept of an eigenvector of a matrix. Our work concerns $l^2$-eigenvectors of Lim or $Z$-eigenvectors of Qi. 

Let $A= (a_{i_1\dots i_d})_{i_j=1}^{n},\, a_{i_1\dots i_d} \in \K{C}$ be an $n$-dimensional tensor of order $d$ ($n^d$-tensor). A non-zero vector $x\in \K{C}^{n}\setminus \{0\}$ is called an \emph{eigenvector} of $A$ if there exists $\lambda \in \K{C}$, the corresponding \emph{eigenvalue}, such that
\begin{align*}
  Ax^{d-1} = \lambda x,\quad Ax^{d-1}:=\left(\sum\limits_{i_2,\dots,i_d = 1}^n a_{1i_2\dots i_d} x_{i_2}\cdots x_{i_d},\,\dots\,,\sum\limits_{i_2,\dots,i_d = 1}^n a_{ni_2\dots i_d} x_{i_2}\cdots x_{i_d}\right).
\end{align*}
For $d=2$ one recovers the definition of an eigenvector of an $n\times n$ matrix $A=(a_{i_1i_2})_{i_j=1}^n$. The point $[x] \in \K{C}P^{n-1}$ defined by an eigenvector $x \in \K{C}^n\setminus \{0\}$ is called an \emph{eigenpoint} and the set of all eigenpoints is called an \emph{eigenconfiguration}.

An $n^d$-tensor $A = (a_{i_1\dots i_d})_{i_j=1}^n,\, a_{i_1,\dots,i_d}\in \K{C}$ is said to be \emph{symmetric} if $a_{i_{\sigma_1}\dots i_{\sigma_d}} = a_{i_1\dots i_d}$ for any permutation $\sigma \in S_d$.
Cartwright and Sturmfels \cite{CarStu} proved that the number of eigenpoints of a generic (symmetric) $n^d$-tensor is equal to 
\begin{align*}
m_{d,n} := \frac{(d-1)^n-1}{d-2} = (d-1)^{n-1}+\dots + (d-1) + 1
\end{align*} 
The same count holds for generic real (symmetric) tensors (i.e. tensors with real entries) but, except for the case of real symmetric matrices ($d=2$), not all eigenvectors of a general real symmetric tensor of order $d\geq 3$  are real. In fact, ``most''\footnote{As it often happens in real algebraic geometry problems the objects of ``maximal complexity'' are rare and ``numerically invisible''.} of real symmetric tensors have eigenpoints in $\K{C}P^{n-1}\setminus \K{R}P^{n-1}$. Abo, Seigal and Sturmfels  conjectured \cite[Conjecture 6.5]{ASS} that for any $n\geq 2, d\geq 1$ there exists a generic real symmetric $n^d$-tensor having only real eigenvectors and proved it for $n=3, d\geq 1$ and for $n=d=4$. The cases $n=2, d\geq 1$ and $n\geq 2, d=2$ are elementary, the case of general $n,d$ was unknown (see for example \cite{Stu}).
In the present work we cover the case of arbitrary $n$ and $d$, proving the following theorem.
\begin{thm}\label{thm 1}
For any $n\geq 2$ and $d\geq 1$ there exists a generic real symmetric $n^d$-tensor all of whose $m_{d,n}$ eigenpoints are real. Equivalently (see Section \ref{hp}), there exists a real homogeneous polynomial of degree $d$ in $n$ variables whose restriction to the sphere $S^{n-1} \subset \K{R}^n$ has the maximum possible finite number of critical points that is equal to $2m_{d,n}$. Moreover, such a symmetric tensor (homogeneous polynomial) exists among traceless tensors (harmonic polynomials).
\end{thm}
In \cite{Gichev} Gichev constructed for any $d\geq 1$ a homogeneous harmonic polynomial of degree $d$ in $3$ variables having $2m_{d,3} = 2(d^2-d+1)$ critical points on the sphere $S^2$. The idea of the proof of Theorem \ref{thm 1} is based on the construction of Gichev.

\subsubsection*{Applications}
Tensors and their eigenvectors arise in many areas of research in pure mathematics and the applied sciences. Below we discuss two problems to which our work is closely related.

\emph{Low rank approximations} (\cite{Chen, FO}). The problem of approximating a general tensor by a tensor of lower rank appears, for example, in signal processing \cite{KR}; see \cite{GK} and references therein for some other applications. Eigenvectors and eigenvalues of a symmetric tensor can be used to find its best rank one approximation. A real symmetric $n^d$-tensor $A=(a_{i_1\dots i_d})_{i_j=1}^n$ is said to be of \emph{rank one} if $a_{i_1\dots i_d} =\lambda x_{i_1}\cdots x_{i_d}$ for some vector $x\in S^{n-1}$ and constant $\lambda\in \K{R}$. Consider the set 
\begin{align*}
  X_{d,n} : = \{ \lambda (x_{i_1}\cdots x_{i_d})_{i_j=1}^n : \lambda \in \K{R}, x\in S^{n-1}\}
\end{align*}
of real symmetric $n^d$-tensors of rank one and for a given real symmetric $n^d$-tensor $A=(a_{i_1\dots i_d})_{i_j=1}^n$ define the function:
\begin{align*}
\txt{dist}_A: X_{d,n} &\rightarrow \K{R}\\  
 \lambda (x_{i_1}\cdots x_{i_d})_{i_j=1}^n &\mapsto \sum\limits_{i_1,\dots, i_d=1}^n(a_{i_1\dots i_d} - \lambda x_{i_1}\cdots x_{i_d})^2
\end{align*}

\noindent
(this function measures the euclidean distance of a rank one tensor from $A$).

A rank one tensor $\lambda(x_{i_1}\cdots x_{i_d})_{i_1,=1}^n\in X_{d,n}$ is a critical point of $\txt{dist}_A$ if and only if $x\in S^{n-1}$ is a unit eigenvector of $A$ and $\lambda\in \K{R}$ is the corresponding eigenvalue. In this context a \emph{best rank one approximation} to $A$, a tensor $\lambda(x_{i_1}\cdots x_{i_d})_{i_j=1}^n\in X_{d,n}$ which is a global minimizer of $\txt{dist}_A$, corresponds to the greatest (in absolute value) eigenvalue $|\lambda|$ \cite[Thm. 2.19]{QiLuo}.

Our Theorem \ref{thm 1} is then equivalent to the existence for any $n\geq 2, d\geq 1$ of a real symmetric $n^d$-tensor $A$ such that the function $\txt{dist}_A: X_{d,n}\rightarrow\K{R}$ has the maximum possible number of critical points that is equal to $m_{d,n}$.

\begin{remark}
The problem of finding a best rank one approximation to a real symmetric $n^d$-tensor $A=(a_{i_1\dots i_d})_{i_j=1}^n$ is equivalent to the problem of maximizing the absolute value $|f_A(x)|$ of the homogeneous polynomial $f_A(x) = \sum_{i_j=1}^n a_{i_1\dots i_d}x_{i_1}\cdots x_{i_d}$ constrained on the sphere $S^{n-1}$.
\end{remark}
\emph{Complex dynamics} (\cite{FS, Rob}). Let $f: \K{C}P^{n-1} \rightarrow \K{C}P^{n-1}$ be a non constant holomorphic map. Then in homogeneous coordinates we can write $f=[f_1: \cdots : f_n ]$, where $$f_i(x)=\sum_{i_2,\dots,i_d=1}^n a_{i i_2\dots i_n} x_{i_2}\cdots x_{i_d},\ i=1,\dots,n$$ are complex homogeneous polynomials of certain degree $d-1$ having no common zeroes in $\K{C}P^{n-1}$. Moreover, the polynomials $f_1,\dots,f_n$ are determined uniquely up to a common constant multiple. It is straightforward to see that the fixed points $\{x\in \K{C}P^{n-1}: f(x) = x\}$ of ${f=[f_1:\cdots:f_n]}$ are precisely the eigenpoints $x\in \K{C}P^{n-1}$ of the tensor ${A=(a_{i_1\dots i_d})_{i_j=1}^n}$ and hence their number for a generic map $f$ is equal to $m_{d,n}$.

When the polynomials $f_1,\dots,f_n$ are real, $f=[f_1:\cdots:f_n]$ preserves $\K{R}P^{n-1}\subset \K{C}P^{n-1}$ and the real fixed points of this map are precisely the real eigenpoints of $A$. Our results imply that for some generic real map $f$ all of its (a priori complex) fixed points are real.
 
\section{Preliminaries}
\subsection{Morse functions}

Let $f$ be a smooth function on a smooth manifold $M$. A critical point $x\in M$ of $f$ is said to be \emph{non-degenerate} if the \emph{Hessian matrix} $\left( \frac{\partial^2 f}{\partial x_i\partial x_j}(x)\right)$ of $f$ at $x$ is non-singular. A smooth function $f: M\rightarrow \K{R}$ with only non-degenerate critical points is called \emph{Morse}. Non-degenerate critical points are isolated, hence on a compact manifold a Morse function can have only finitely many critical points.

\subsection{From symmetric tensors to homogeneous polynomials}\label{hp}
A generic symmetric $n\times n$ matrix has $n$ real simple eigenvalues and $n$ corresponding  eigenpoints. Moreover, in the space of all symmetric $n\times n$ matrices those which have repeated eigenvalues form a real algebraic subvariety, that we call the \emph{discriminant}, and a generic matrix belongs to its complement. The codimension of the discriminant is two and this justifies the fact that the number of real eigenpoints is the same for all generic matrices.

Let $A\hspace{-1pt}=\hspace{-2pt} (a_{i_1\dots i_d})_{i_j=1}^n,\, a_{i_1\dots i_d}\hspace{-3pt} \in\hspace{-3pt} \K{R}$ be an $n$-dimensional symmetric tensor of order $d$. Recall that a complex number $\lambda \in \K{C}$ is an \emph{eigenvalue} associated to an eigenvector $x\in \K{C}^n$ if $Ax^{d-1} = \lambda x$. In this case the pair $(x,\lambda)\in \K{C}^n\setminus\{0\}\times \K{C}$ is called an \emph{eigenpair} of $A$. Two eigenpairs $(x,\lambda)$ and $(x^{\prime},\lambda^{\prime})$ of $A$ are said to be equivalent if they define the same eigenpoint $[x]=[x^{\prime}] \in \K{C}P^{n-1}$. Theorem $1.2$ in \cite{CarStu} asserts that the number of eigenpoints (equivalence classes of eigenpairs) of a generic symmetric $n^d$-tensor is equal to $m_{d,n} = ((d-1)^n-1)/(d-2) = (d-1)^{n-1}+\dots+(d-1)+1$. \emph{Non-generic} tensors are cut out by an algebraic hypersurface, called the \emph{eigendiscriminant} \txt{\cite{ASS}}, and the number of eigenpoints of a non-generic tensor is not equal to the expected $m_{d,n}$. On each connected component of the complement of the eigendiscriminant the number of real eigenpoints (equivalence classes of real eigenpairs) is constant.

There is a well-known one-to-one correspondence between the set $\mathcal{P}_{d,n}$ of real homogeneous polynomials of degree $d$ in $n$ variables and the set of real symmetric $n^d$-tensors:
\begin{align}\label{corr}
f_A = \sum\limits_{i_1,\dots,i_d=1}^n a_{i_1\dots i_d}x_{i_1}\dots x_{i_d} \quad \longleftrightarrow\quad A = (a_{i_1\dots i_d})_{i_j=1}^n
\end{align}
The critical points of the restriction $f_A|_{S^{n-1}}$ of a homogeneous polynomial $f_A$ to the unit sphere are precisely unit real eigenvectors of the corresponding symmetric tensor $A$. Indeed, by the method of Lagrange multipliers, if $x\in S^{n-1}$ then
\begin{align*}
\txt{d}_x f_A|_{S^{n-1}} = 0 \quad \Leftrightarrow \quad  \txt{d}_x f_A = \lambda\, \txt{d}_x \left(\frac{\Vert x \Vert^2-1}{2}\right) \quad \Leftrightarrow \quad  Ax^{d-1} = (\lambda/d) x
\end{align*}
Note that the Lagrange multiplier $\lambda$ corresponds to the eigenvalue $\lambda/d$ associated to the unit eigenvector $x$. In the terminology of Lim \cite{Lim} and Qi \cite{Qi} unit real eigenvectors are $l^2$-eigenvectors and $Z$-eigenvectors respectively. Theorem \cite[Thm.\,1.2.]{CarStu} thus gives an upper bound on the number of critical points of the restriction of a homogeneous polynomial to the sphere.
\begin{lemma}\label{bound}
  If a polynomial $f\in \mathcal{P}_{d,n}$ defines a Morse function $f|_{S^{n-1}}$ on the sphere then the number of critical points of $f|_{S^{n-1}}$ is bounded by $2m_{d,n} = 2((d-1)^n-1)/(d-2)$.  
\end{lemma}
\begin{proof}
 If $f_A|_{S^{n-1}}$ is a Morse function and the tensor $A$ is generic then $A$ has $m_{d,n}$ eigenpoints in $\K{C}P^{n-1}$ which implies that the number of unit real eigenvectors of $A$ (that is equal to the number of critical points of $f_A|_{S^{n-1}}$) is bounded by $2m_{d,n}$. 
 
 Suppose now that $f_A|_{S^{n-1}}$ is a Morse function but the tensor $A$ is not generic. Since non-generic tensors form a hypersurface in the space of symmetric tensors any open neighbourhood of $A$ contains a generic tensor $\tilde{A}$. Moreover, if $\tilde A$ is sufficiently close to $A$ by \cite[Cor. 5.24]{BanHur} the function $f_{\tilde A}|_{S^{n-1}}$ is Morse and it has the same number of critical points as $f_{A}|_{S^{n-1}}$. 
 \end{proof}

\subsection{Spherical harmonics}
Consider the space 
$$\mathcal{H}_{d,n} = \left\{f|_{S^{n-1}}: f \in \mathcal{P}_{d,n},\ \frac{\partial^2f}{\partial x_1^2}+\cdots+\frac{\partial^2f}{\partial x_n^2} = 0\right\}$$
of restrictions to the sphere $S^{n-1}$ of homogeneous harmonic poynomials of degree $d$. 
Note that a polynomial $f_A\in \mathcal{P}_{d,n}$ is harmonic if and only if the symmetric tensor $A$ is \emph{traceless}, i.e.
$$\sum\limits_{i=1}^n a_{ii\, i_{3} \dots i_{d}} = 0\quad \forall\, i_3,\dots, i_d =1,\dots, n$$
It is well-known that $\mathcal{H}_{d,n}$ is the eigenspace of the spherical Laplace operator $\Delta_{S^{n-1}}$ corresponding to the eigenvalue $-d(d+n-2)$. Functions in $\mathcal{H}_{d,n}$ are called \emph{spherical harmonics} of degree $d$ and the dimension of $\mathcal{H}_{d,n}$ is equal to
$$\txt{dim}\, \mathcal{H}_{d,n} = {n+d-1 \choose d} - {n+d-3 \choose d-2}\quad \txt{if} \quad d\geq 2\quad \txt{and} \quad \txt{dim}\, \mathcal{H}_{0,n} = 1,\quad \txt{dim}\,\mathcal{H}_{1,n} = n$$ 

For any point $y\in S^{n-1}$ and any $d$ there exists a spherical harmonic $Z^{y}_d \in \mathcal{H}_{d,n}$, called \textit{zonal}, which is invariant under rotations preserving $y$:
$$Z_d^{y}(Rx) = Z_d^{y}(x),\quad R \in \txt{SO}(n),\ Ry=y$$
The function $Z_d^y(x)$ is determined uniquely up to a constant and is proportional to $G_{d,n}(\langle x,y\rangle)$ \cite[Thm. 2.14]{SteWei},\footnote{According to the usual definition \cite[page 143]{SteWei} a zonal harmonic $Z_d^y$ is determined uniquely by some normalization condition. Since a normalization is unimportant for our purposes we abuse the terminology and call zonal any spherical harmonic with the mentioned invariance property.} where $\langle x, y\rangle = x_1y_1+\dots +x_ny_n$ is the standard scalar product in $\K{R}^n$ and $G_{d,n}$ is the \emph{Gegenbauer polynomial of degree $d$ and parameter $\frac{n-2}{2}$}. The polynomials $\{G_{d,n}\}_{d\geq 0}$ can be defined by the recurrence relation \cite[22.4.2, 22.7.3]{AbrSte}:
\begin{align*}
  G_{0,n}(x) &= 1,\\
  G_{1,n}(x) &= (n-2)x,\\
  G_{d,n}(x) &= \frac{1}{d}\left[2x(d+\frac{n}{2}-2)G_{d-1,n}(x) - (d+n-4) G_{d-2,n}(x) \right]
\end{align*}
and they form an orthogonal family on the interval $[-1,1]$ with respect to the measure $(1-z^2)^{\frac{n-3}{2}}\,dz$ \cite[22.2.3]{AbrSte}:
$$\int\limits_{-1}^1 G_{d_1,n}(z)\, G_{d_2,n}(z)\, (1-z^2)^{\frac{n-3}{2}} dz = 0,\quad d_1 \neq d_2$$
Therefore by \cite[Prop. I.1.1]{Far} $G_{d,n}$ has $d$ simple real roots in $(-1,1)$ and hence its derivative $G^{\prime}_{d,n}$ has $d-1$ roots in $(-1,1)$ which we denote by $\alpha_{d,1},\dots,\alpha_{d,d-1}$. 
The following lemma characterizes the critical points of a zonal spherical harmonic.
\begin{lemma}\label{zonal}
The set of critical points of $Z_d^y$ consists of $y, -y$ and $d-1$ affine hyperplane sections of the sphere
 $\{x\in S^{n-1}: \langle x, y \rangle = \alpha_{d,i}\},\ i=1,\dots,d-1$. The critical points $y$ and $-y$ are non-degenerate. 
\end{lemma}
\begin{proof}
A point $x\in S^{n-1}$ is critical for $G_{d,n}(\langle x,y \rangle)$ if and only if $G^{\prime}_{d,n}(\langle x,y\rangle)y$ is proportional to $x$. This is possible either if $\langle x,y\rangle$ is a root of $G^{\prime}_{d,n}$ or $x=\pm y$. To prove the non-degeneracy of $x=\pm y$ we assume without loss of generality that $y=(0,\dots,0,1)\in S^{n-1}$ and then in local coordinates 
$$(x_1,\dots,x_{n-1}) \mapsto \left(x_1,\dots,x_{n-1},\pm \sqrt{1-x_1^2-\dots-x_{n-1}^2}\right) \in S^{n-1}$$ around $x=\pm y$ our function $G_{d,n}(\langle x,y\rangle)$  takes the form $G_{d,n}\left(\pm \sqrt{1-x_1^2-\dots-x_{n-1}^2}\right)$. One can easily verify that its Hessian matrix at $(x_1,\dots,x_{n-1})=(0,\dots,0)$ is non-singular.
\end{proof}
The inclusion map
\begin{align*}
  i : \mathcal{P}_{d,n} &\hookrightarrow \mathcal{P}_{d,n+1}\\
  f &\mapsto i(f)(x_1,\dots,x_n,x_{n+1}) = f(x_1,\dots,x_{n})
\end{align*}
induces the linear inclusion
\begin{align*}
  \mathlarger{\hat{}}\,: \mathcal{H}_{d,n} &\hookrightarrow \mathcal{H}_{d,n+1}\\
  h = f|_{S^{n-1}} &\mapsto \hat{h} = i(f)|_{S^{n}}
\end{align*}
The critical points of $\hat{h}$ are described as follows.
\begin{lemma}\label{sectoral}
Assume that $d,n \geq 2$ and $h \in \mathcal{H}_{d,n}$. 
\begin{enumerate}[label=(\roman*)]
\item If the zero locus $\{h=0\} \subset S^{n-1}$ is regular then the set of critical points of $\hat{h} \in \mathcal{H}_{d,n+1}$ consists of $\pm(0,\dots,0,1) \in S^{n}$ and the points $(x_1,\dots,x_n,0)$, where $(x_1,\dots,x_n) \in S^{n-1}$ is critical for $h$. Moreover, for $d\geq 3$ the points $\pm(0,\dots,0,1)\in S^n$ are always degenerate.

\item If $\{h=0\}$ is singular then, additionally, for each singular point $(x_1,\dots,x_{n}) \in \{h=0\}$ the great circle $\{(tx_1,\dots,tx_n,\pm\sqrt{1-t^2}):\, 0\leq t \leq 1\} \subset S^n$ consists of critical points of $\hat{h}$. 
\end{enumerate}
\end{lemma}
\begin{proof}
If $h = f|_{S^{n-1}}$ for some harmonic polynomial $f \in \mathcal{P}_{d,n}$ the critical points of $\hat{h}=i(f)|_{S^{n}}\in \mathcal{H}_{d,n+1}$ are characterized by
\begin{align}\label{system for sectoral}
\frac{\partial f}{\partial x_1} = \lambda x_1, \quad \dots \quad \frac{\partial f}{\partial x_n} = \lambda x_n, \quad \frac{\partial f}{\partial x_{n+1}} = 0 = \lambda x_{n+1}
\end{align}
Obviously $(x_1,\dots,x_n,0) \in S^n$ is a critical point of $\hat{h}$ if $(x_1,\dots,x_n)\in S^{n-1}$ is critical for $h$. Now if $\lambda=0$ and $\{h=0\}\subset S^{n-1}$ is regular then $x_1=\dots=x_n=0$ and $x_{n+1}=\pm 1$. If, instead, $\{h=0\}$ is singular and $(x_1,\dots,x_n) \in \{h=0\}$ is a solution of $\frac{\partial f}{\partial x_1} = \dots = \frac{\partial f}{\partial x_n} = 0$ then due to the homogeneity of $f$ any point $(tx_1,\dots,tx_n, \pm\sqrt{1-t^2}),\, 0\leq t \leq 1$ is a solution of the system \eqref{system for sectoral} with $\lambda=0$.
\end{proof}
\section{Proof of Theorem \ref{thm 1}}
Denote by $Z_{d,n}$ a zonal spherical harmonic $Z_d^y(x) = G_{d,n}(\langle x, y\rangle) = G_{d,n}(x_n) \in \mathcal{H}_{d,n}$ corresponding to the point $y=(0,\dots,0,1) \in S^{n-1}$ and let $M_{d,n}\in \mathcal{H}_{d,n}$ be any Morse spherical harmonic with the maximum possible number of critical points. Note that by Lemma \ref{bound} this number is bounded by $2m_{d,n} = 2((d-1)^n-1)/(d-2)$. In dimension $n=2$ any $h \in \mathcal{H}_{d,2}$ is just a trigonometric polynomial $$h = a \cos (d\theta) + b \sin (d\theta),\quad a,b \in \K{R},\quad \theta \in [0,2\pi)$$ and hence it is a Morse function on $S^1$ with $2m_{d,2}=2d$ critical points. For $n, d\geq 3$ the number of critical points of a general spherical harmonic $h \in \mathcal{H}_{d,n}$ is not anymore a constant and depends significantly on the choice of $h$. In the proposition below we exhibit for any $n,d\geq 2$ a Morse spherical harmonic $M_{d,n}\in \mathcal{H}_{d,n}$ having $2m_{d,n}$ critical points. In fact, we construct $M_{d,n}$ by induction on $n$ starting from a trigonometric polynomial $M_{d,2}\in \mathcal{H}_{d,2}$.  
\begin{propo}\label{propo 1}
  For any $d, n\geq 2$  and a sufficiently small $\varepsilon>0$ the spherical harmonic $M_{d,n+1} := Z_{d,n+1} + \varepsilon\, \hat{M}_{d,n} \in \mathcal{H}_{d,n+1}$ is a Morse function on $S^{n}$ with $2m_{d,n+1}$ critical points.
\end{propo}
\begin{proof}
As observed above one can take $M_{d,2} = a \cos(d \theta) + b \sin(d \theta)$. Suppose that for some $n\geq 2$, we have already constructed a Morse spherical harmonic $M_{d,n}\in \mathcal{H}_{d,n}$ with $2m_{d,n}$ critical points on $S^{n-1}$. By Lemmas \ref{zonal} and \ref{sectoral} we have that the points $\pm(0,\dots,0,1) \in S^n$ are critical for both $Z_{d,n+1}$ and $\hat{M}_{d,n}$ and hence also for the perturbation $Z_{d,n+1}+\varepsilon \hat{M}_{d,n}$. Since the points $\pm(0,\dots,0,1)$ are non-degenerate for $Z_{d,n+1}$ they remain non-degenerate for the perturbation for small enough $\varepsilon>0$.

We prove that each of the $d-1$ critical circles $\{x\in S^{n}: \langle x, y \rangle = \alpha_{d,i}\},\ i=1,\dots,d-1$ of $Z_{d,n+1}$ breaks into $2m_{d,n}$ non-degenerate critical points when $Z_{d,n+1}$ is slightly perturbed by $\hat{M}_{d,n}$. The idea is shown on Figure \ref{fig:pert}, where the red/purple color represents positive/negative values of functions. In spherical coordinates
\begin{align*}
  x_1 &= \sin \theta_n\cdot \tilde{x}_1 = \sin \theta_{n} \sin \theta_{n-1} \cdots \sin \theta_2 \sin \theta_1\\
  x_2 &= \sin \theta_n \cdot \tilde{x}_2 = \sin \theta_{n} \sin \theta_{n-1} \cdots \sin \theta_2 \cos \theta_1\\
  x_3 &= \sin \theta_n \cdot \tilde{x}_3 = \sin \theta_{n} \sin \theta_{n-1} \cdots \cos \theta_2\\
  &\ \ \vdots\\
  x_n &= \sin \theta_n \cdot \tilde{x}_n = \sin \theta_n \cos \theta_{n-1}\\
  x_{n+1} &= \cos \theta_{n}
\end{align*}
on $S^n$, where $(\tilde{x}_1,\dots,\tilde{x}_n) \in S^{n-1}=\{x \in S^n: x_{n+1} =0\}$, we have 
\begin{align*}
Z_{d,n+1}(x_1,\dots,x_{n+1}) &= G_{d,n+1}(x_{n+1}) = G_{d,n+1}(\cos \theta_n)\\
\hat{M}_{d,n}(x_1,\dots,x_{n+1}) &= \sin^d \theta_n\, M_{d,n}(\tilde{x}_1,\dots, \tilde{x}_n)
\end{align*}
\begin{figure}[!htb]\label{fig:pert}
\minipage{0.27\textwidth}
  \includegraphics[width=\linewidth]{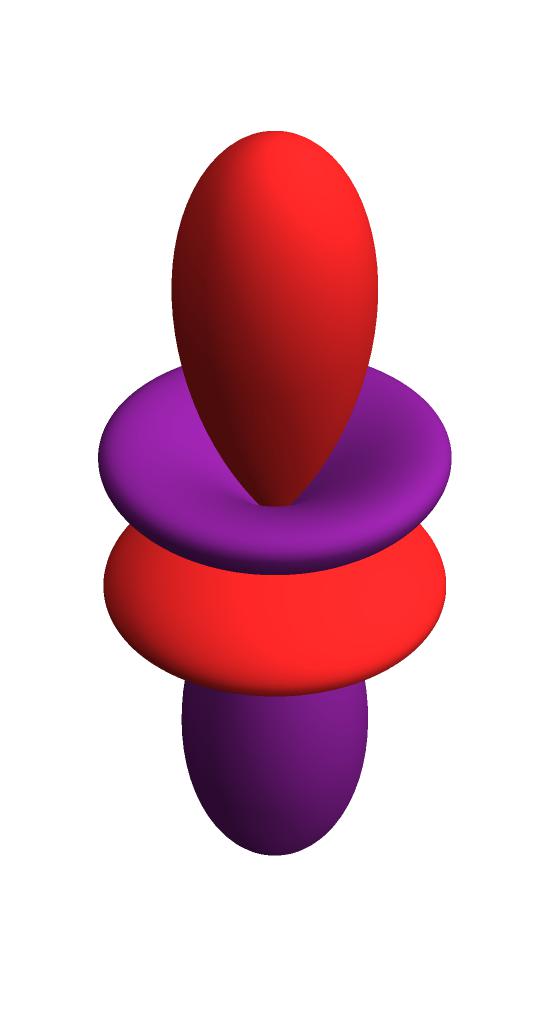}
\endminipage\hfill
\minipage{0.32\textwidth}
\hbox{\hspace{-0.1cm}  \includegraphics[width=1.3\textwidth]{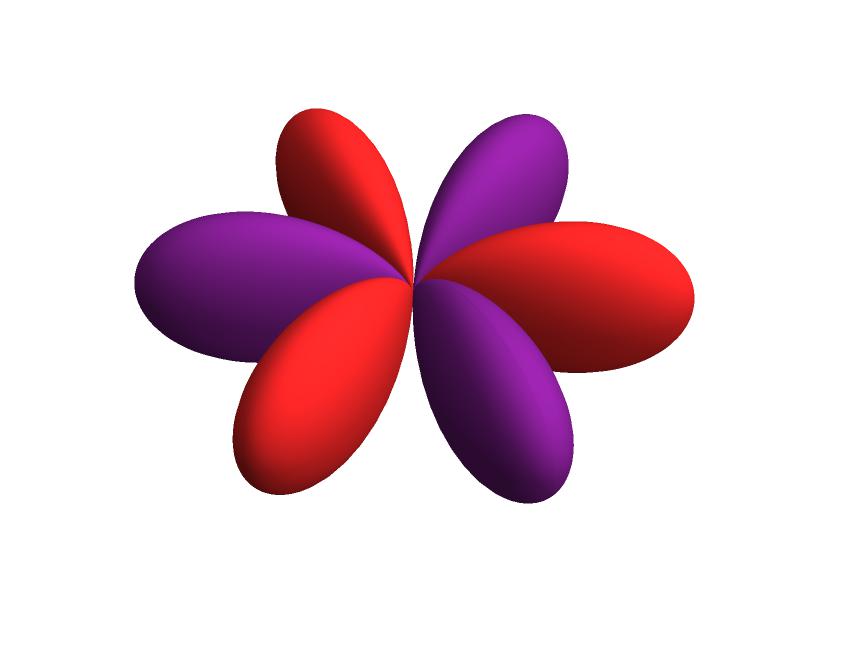}}
\endminipage\hfill
\minipage{0.312\textwidth}%
\vbox{\vspace{0.25cm}  \includegraphics[width=\linewidth]{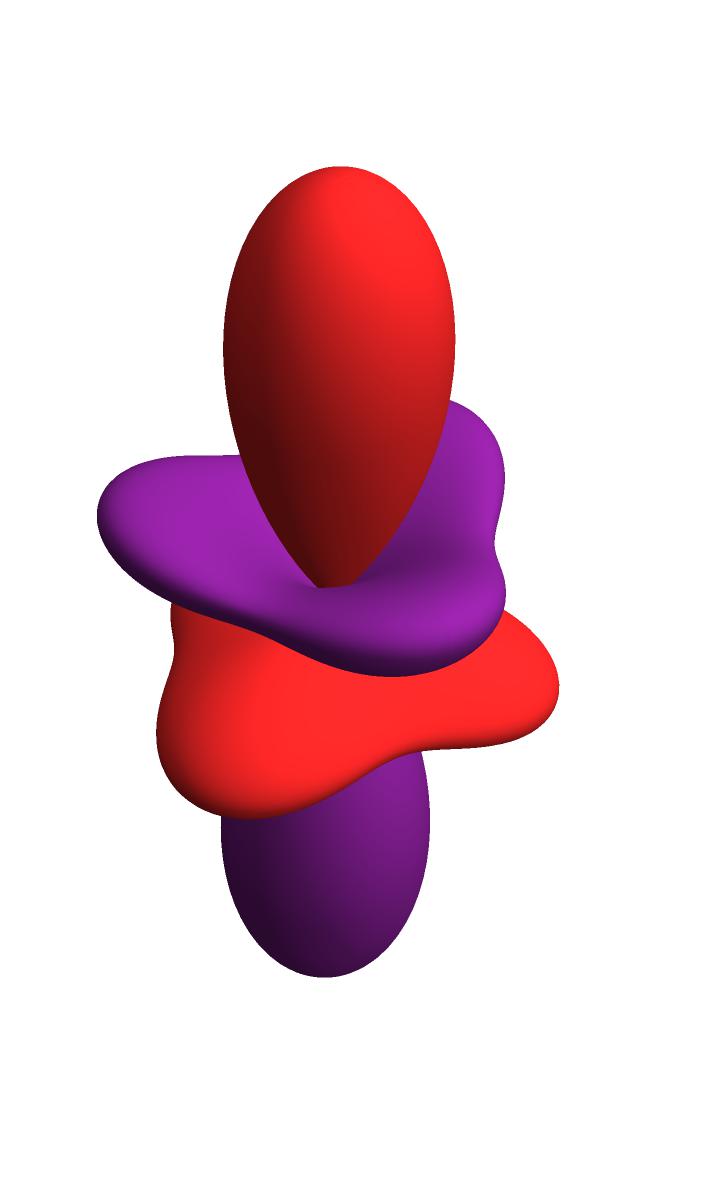}}
\endminipage
\vspace{-0.5cm}
\caption{Plots of the absolute values on the sphere $S^2$ of a zonal harmonic $Z_{3,3}$ (left), a function $\hat{M}_{3,2}$ with $6$ critical points on the circle (middle) and a perturbation of $Z_{3,3}$ by $\hat{M}_{3,2}$ with $14=2+2\cdot 6$ non-degenerate critical points (right).}
\end{figure}

\noindent
and hence the critical points of $Z_{d,n+1}+\varepsilon \hat{M}_{d,n}$ are described by the equations
\begin{align}
\varepsilon\, \sin^d \theta_n\, \frac{\partial}{\partial \theta_1} M_{d,n}(\tilde{x}_1,\dots,\tilde{x}_n) &= 0\nonumber \\
&\vdots\nonumber \\
\varepsilon\, \sin^d \theta_n\, \frac{\partial}{\partial \theta_{n-1}} M_{d,n}(\tilde{x}_1,\dots,\tilde{x}_n) &= 0\nonumber \\
\frac{\partial}{\partial \theta_n}\left[ G_{d,n+1}(\cos \theta_n) + \varepsilon\, \sin^{d} \theta_n \,M_{d,n}(\tilde{x}_1,\dots,\tilde{x}_n) \right]&= 0 \label{last eq}
\end{align}
Since the $d-1$ zeroes of $G^{\prime}_{d,n+1}$ are non-degenerate, then for a fixed $(\tilde{x}_1,\dots,\tilde{x}_n) \in S^{n-1}$ the equation \eqref{last eq} has $d-1$ non-degenerate solutions provided that $\varepsilon$ is small enough.
It follows that each critical point $(\tilde{x}_1,\dots,\tilde{x}_n) \in S^{n-1}$ of $M_{d,n}$ gives rise to $d-1$ critical points of $Z_{d,n+1} + \varepsilon \hat{M}_{d,n}$.
In spherical coordinates the Hessian matrix of $Z_{d,n+1} + \varepsilon \hat{M}_{d,n}$ computed at a critical point $\theta =(\theta_1,\dots,\theta_{n-1},\theta_n)$ has the block-diagonal form:
\begin{align*}
  \begin{pmatrix}
    \varepsilon\, \sin^d \theta_n \frac{\partial^2 M_{d,n}}{\partial \theta_1^2}(\theta)  & \dots & \varepsilon\, \sin^d \theta_n \frac{\partial^2 M_{d,n}}{\partial \theta_1\partial \theta_{n-1}}(\theta)  & 0 \\
    \vdots & \ddots & \vdots & \vdots \\
    \varepsilon\, \sin^d \theta_n \frac{\partial^2 M_{d,n}}{\partial \theta_{n-1}\partial \theta_{1}}(\theta) & \dots & \varepsilon\, \sin^d \theta_n \frac{\partial^2 M_{d,n}}{\partial \theta^2_{n-1}}(\theta)  & 0 \\ 
 0 & \dots & 0 & \frac{\partial^2}{\partial \theta_n^2} \left[ G_{d,n+1}(\cos \theta_n) + \right.\\
&&&\hspace{1cm} \left.\varepsilon\, \sin^{d} \theta_n \,M_{d,n}(\tilde{x}_1,\dots,\tilde{x}_n) \right]
  \end{pmatrix}
\end{align*}
It is non-singular since the function $M_{d,n}$ is, by assumption, Morse and for a small $\varepsilon$ the solutions of \eqref{last eq} are non-degenerate. Thus, the function $Z_{d,n+1}+\varepsilon \hat{M}_{d,n}$ has $2+(d-1)\cdot 2((d-1)^n-1)/(d-2) = 2((d-1)^{n+1}-1)/(d-2)=2m_{d,n+1}$ non-degenerate critical points.
\end{proof}
\noindent
Theorem \ref{thm 1} follows from the proposition. 
\begin{remark}
  The described construction of homogeneous polynomials (spherical harmonics) with maximum finite number of critical points on the sphere can be generalized as follows. Instead of a spherical harmonic $M_{d,n}$ one can take any homogeneous polynomial $f=f(x_1,\dots,x_n)\in \mathcal{P}_{d,n}$ having $2m_{d,n}$ critical points on $S^{n-1}$ and instead of the Gegenbauer polynomial $G_{d,n+1}=G_{d,n+1}(x_{n+1})$ one can take any even (for even $d$) or odd (for odd $d$) univariate degree $d$ polynomial $p=p(x_{n+1})$ whose derivative $p^{\prime} = p^{\prime}(x_{n+1})$ has $d-1$ simple roots in $(-1,1)$. Then for a small $\varepsilon$ the function $p(x_{n+1})+\varepsilon f(x_1,\dots,x_n) \in \mathcal{P}_{d,n+1}|_{S^{n}}$ has $2m_{d,n+1}$ critical points on $S^{n}$. 
\end{remark}

\section*{Acknowledgements}
I am grateful to Andrei Agrachev and Antonio Lerario for fruitful discussions, to Bernd Sturmfels for useful comments and suggestions and to Shamil Asgarli for reading the manuscript and giving a feedback.

\bibliographystyle{plain}

\bigskip{\footnotesize%
  \textsc{SISSA, via Bonomea 265, 34136 Trieste, Italy}\par  
  \textit{E-mail address}: \texttt{kkozhasov@sissa.it}}
\end{document}